\documentclass[a4paper, 11pt]{article}

\usepackage{amsmath,amsthm,amssymb,enumitem,bbm}
\usepackage[margin=2cm]{geometry}

\usepackage[capitalise]{cleveref}
\usepackage[colorinlistoftodos,prependcaption,color=green!5,textsize=footnotesize]{todonotes}

\usepackage{thmtools} 
\usepackage{thm-restate}

\theoremstyle{plain}
\newtheorem{theorem}{Theorem}

\newtheorem{conjecture}[theorem]{Conjecture}

\newtheorem{definition}[theorem]{Definition}

\newtheorem{lemma}[theorem]{Lemma}

\newtheorem{observation}[theorem]{Observation}

\theoremstyle{definition}

\newtheorem*{claim*}{Claim}

\newcommand{\D}{\mathcal{D}}
\newcommand{\F}{\mathcal{F}}

\newcommand{\fdn}{FD}

\title{Graphs with minimum fractional domatic number}
\author{
	Maximilien Gadouleau
	\thanks{Department of Computer Science, Durham University, UK. Email: \texttt{m.r.gadouleau@durham.ac.uk}}  
	\and Nathaniel Harms \thanks{EPFL, Switzerland. Email: \texttt{nathaniel.harms@epfl.ch}. This work was done while the author was visiting the University of Liverpool. Partly funded by an NSERC MSFSS award.}
	\and George B. Mertzios
	\thanks{Department of Computer Science, Durham University, UK. Email: \texttt{george.mertzios@durham.ac.uk}} 
	\and Viktor Zamaraev
	\thanks{Department of Computer Science, University of Liverpool, UK. Email: \texttt{viktor.zamaraev@liverpool.ac.uk}}
}

\begin{document}

\maketitle

\begin{abstract}
	The domatic number of a graph is the maximum number of vertex disjoint dominating sets that partition the vertex set of the graph. In this paper we consider the fractional
	variant of this notion.
	Graphs with fractional domatic number 1 are exactly the graphs that contain an isolated vertex. 
	Furthermore, it is known that all other graphs have fractional domatic number at least 2.
	In this note we characterize graphs with fractional domatic number 2.
	More specifically, we show that a graph without isolated vertices has fractional domatic number
	2 if and only if it has a vertex of degree 1 or a connected component isomorphic to a 4-cycle.
	We conjecture that if the fractional domatic number is more than 2, then it is at least 7/3. 
\end{abstract}

\section{Introduction}

A set of vertices in a graph $G$ is \emph{dominating} if every vertex of the graph either belongs to the set or has a neighbour in the set. 
A \emph{domatic partition} of $G$ is a partition of its vertices into dominating sets.
The maximum number of sets in a domatic partition of $G$ is called the \emph{domatic number} of $G$. This concept was introduced in 1975 by Cockayne and Hedetniemi \cite{CH75}, and since then appeared in many studies. It arises in several applications including facility location in networks \cite{FYK00} and lifetime maximization of sensor networks \cite{MW05}.

In sensor network applications, a typical scenario is that small battery-powered devices are deployed in a remote
area where they need to continuously monitor environmental conditions (e.g. temperature, pressure, etc.) via sensors. The energy limitations of the devices and the remoteness of the network demand efficient power management. Dominating set based scheduling turned out to be a useful concept in this context.
The redundancy graph consists of the vertices corresponding to the devices where two devices are connected
by an edge if they can monitor the same area, i.e. one of them can be asleep when the other is active and vice  versa. Thus, in order to monitor the entire area at any given moment it is sufficient that only nodes of
a dominating set of the redundancy graph are active while the other nodes might be in a power saving mode.
As an example let us consider a network that consists of five devices $A,B,C,D,E$ each having one-month long battery and the redundancy graph is a 5-cycle $(A,B,C,D,E,A)$. If no sleeping schedule is applied and all devices are always active, such a network could serve at most one month. A more efficient approach is to partition the vertices into two dominating sets, say $\{A,B,D\}$ and $\{C,E\}$, and let only the devices in the first set to monitor the area for one month and then only the devices in the second set to do the job in the second month. 
Such a scheduling mechanism doubles the lifetime of the network. 
The domatic number of the redundancy graph is the maximum number of dominating sets that can successively monitor the network.

It turns out that one can achieve a longer lifetime by scheduling not necessarily disjoint dominating sets.
In our example, we can attain a network lifetime of 2.5 months by activating
devices in the following five dominating sets $\{A,C\}, \{B,D\}, \{C,E\}, \{A,D\}, \{B,E\}$ 
in turn for half a month each. The limits of such schedules are characterized by the \emph{fractional} domatic number.
This notion was formally introduced in 2006 by Suomela \cite{Suo06} in the context of lifetime maximization of sensor networks, although the concept was studied  in 2000 by Fujita, Yamashita, and Kameda \cite{FYK00}. The fractional domatic number of a graph $G$ can be defined as follows. For a natural number $s$, let $f(G,s)$ be the maximum number of not necessarily distinct dominating sets such that every vertex is contained in at most $s$ of them. Then the fractional domatic number of $G$, denoted by $\fdn(G)$, is the supremum of $f(G,s)/s$ over all natural numbers $s$. Since the fractional domatic number can alternatively be defined as the solution to the linear programming relaxation of the integer linear program defining domatic number, the supremum is always attained, and hence it can be replaced with `maximum' in the previous sentence (see \cite{SU11} for more detail).

Clearly, if $G$ has an isolated vertex, then it belongs to every dominating set of $G$, and therefore $\fdn(G) = 1$. On the other hand, if $G$ has no isolated vertices,
any maximal independent set and its complement are vertex disjoint dominating
sets of $G$, and hence $\fdn(G)$ is at least 2.

Motivated by sensor network applications Abbas, Egerstedt, Liu, Thomas, and Whalen \cite{WMCRP16} studied fractional domatic number
of $K_{1,6}$-free graphs. The choice of the graph class was motivated by the fact that these graphs include all unit disk graphs, which are often used to model communication graphs of wireless networks. The authors showed that, except eight small graphs, any $K_{1,6}$-free graph with minimum degree at least 2 has fractional domatic number at least 5/2.

In this note we characterize graphs with fractional domatic number 2. More specifically, we prove the following

\begin{theorem}[Main]\label{th:main-intro}
	A graph without isolated vertices has fractional domatic number 2 if and only if it contains a vertex of degree one or a connected component isomorphic to a 4-cycle.
\end{theorem}

It follows that the fractional domatic number of a graph is strictly greater than 2 if and only if the minimum degree of the graph is at least two and every connected component is different from a 4-cycle.

We prove the main result in \cref{sec:main} and conclude the paper in \cref{sec:conclusion}. In the next section we introduce necessary notions and auxiliary results.

\section{Preliminaries}\label{sec:prelim}

We consider simple graphs, i.e. undirected graphs, with no loops or multiple edges.
We denote by $V(G)$ and $E(G)$ the sets of vertices and edges of a graph $G$, respectively.
For a vertex $v$ of a graph $G=(V,E)$ we denote by $N(v)$ the \emph{neighbourhood} of $v$,
i.e. the set of vertices adjacent to $v$, and by $\delta(G)$ the minimum degree of a vertex in $G$.
As usual, $C_k$ and $K_{p,q}$ denote a $k$-vertex cycle and respectively a complete bipartite graph with $p$ and $q$ vertices in the parts.
A \emph{cut vertex} is a vertex in a connected graph that 
disconnects the graph upon deletion.
Similarly, a \emph{cut edge} is an edge in a connected graph 
that disconnects the graph upon deletion.
A graph is \emph{2-connected} if it is connected and contains 
at least 3 vertices, but no cut vertex.
A set $D \subseteq V$ is a \emph{dominating set} of $G$ if every vertex in $V \setminus D$
has a neighbour in $D$. 

\begin{observation}[Ore \cite{Ore62}]
	\label{obs:domComp}
	If $G$ is a graph without isolated vertices then the complement of a 
	minimal dominating set of $G$ is also a dominating set of $G$.
\end{observation}

We say that a multiset $\D = \{ D_1, D_2, \ldots, D_k \}$ of $k \geq 1$ dominating sets 
in $G$ is a \emph{$(k,s)$-configuration} of $G$, if every vertex of $G$ belongs to (we shall also say is covered by) at most $s$
dominating sets in $\D$. The \emph{fractional domatic number} $\fdn(G)$ of $G$ is the maximum of $k/s$ over
all natural numbers $k$ and $s$ such that $G$ admits a $(k,s)$-configuration.

\begin{observation}\label{obs:lowerFDN}
	If a graph $G$ admits a $(k,s)$-configuration, then $\fdn(G) \geq \frac{k}{s}$.
\end{observation}

We will make use of the following known facts about the fractional domatic number 

\begin{lemma}[\cite{GH21}]
	\label{lem:FDCycles}
	For any natural $n$,
	\[
	\fdn(C_n) = \left\{ 
	\begin{array}{cr}
		3, & ~~~\text{if}~~~n \equiv 0 {\pmod {3}} \\
		\frac{3n}{n+2}, & ~~~\text{if}~~~n \equiv 1 {\pmod {3}}\\
		\frac{3n}{n+1}, & ~~~\text{if}~~~n \equiv 2 {\pmod {3}}
	\end{array}
	\right.
	\]
\end{lemma}

\begin{lemma}[\cite{GH21}]
	\label{lem:FDDisjUnion}
	Let $G$ and $H$ be two graphs on disjoint vertex sets.
	Then $\fdn(G \cup H) = \min\{ \fdn(G), \fdn(H) \}$.
\end{lemma}

A \emph{path} in a graph $G=(V,E)$ is a sequence $(v_1, v_2, \ldots, v_k)$
of pairwise distinct vertices, where $v_{i}v_{i+1} \in E$ for every $i = 1, 2, \ldots,k-1$.
The first and the last vertices of a path are called the \emph{end-vertices} of the path, and
all other vertices are the \emph{internal} vertices of the path. A path in $G$ is called \emph{binary}
if all its internal vertices have degree 2 in $G$.
A \emph{cycle} in $G$ is a sequence $(v_1, v_2, \ldots, v_k, v_1)$,
where $(v_1, v_2, \ldots, v_k)$ is a path and $v_kv_1 \in E$.
Given a path or a cycle $H$ in $G$ we denote by $E(H)$ the set of edges of $H$.

Let $H_1$ and $H_2$ be two graphs that have at most one vertex in common. A graph $G$ is called a \emph{$(H_1,H_2)$-dumbbell} if $G$ is obtained from $H_1$ and $H_2$ by connecting
them with a path. More formally, $G$ is the union of $H_1$, $H_2$, and $P$, where
\begin{enumerate}
	\item $P$ is a binary path in $G$ with $|V(P)| \geq 1$;
	\item for each $i \in \{1,2\}$, the set $V(H_i) \cap V(P)$ has exactly one element, which is an end-vertex of $P$;
	\item if $|V(P)| \geq 2$, then $H_1$ and $H_2$ are vertex disjoint.
\end{enumerate}
The path $P$ is called the \emph{handle} of the dumbbell and the graphs $H_1$ and $H_2$
are its \emph{plates}. A graph is a \emph{dumbbell} if it is a $(H_1,H_2)$-dumbbell for some 
$H_1$ and $H_2$.


\begin{lemma}
	\label{lem:deg2decomposition}
	Let $G=(V,E)$ be a connected graph with $\delta(G) \geq 2$. Then
	\begin{enumerate}
		\item $G$ is 2-connected; or
		\item $G$ is a $(H_1,H_2)$-dumbbell for some connected graphs $H_1$ and $H_2$ with 
		$\delta(H_1) \geq 2$, $\delta(H_2) \geq 2$.
	\end{enumerate}
\end{lemma}
\begin{proof}
	If $G$ is 2-connected then we are done. Otherwise $G$ contains a cut vertex. We consider two cases.
	
	The first case is when $G$ has a cut vertex of degree 2. Let $v$ be such a vertex. The graph $G-v$
	has exactly two connected components. One of these components contains one neighbour of $v$ and the other component
	contains the other neighbour. Note that $v$ together with its two neighbours form a binary path in $G$. 
	We extend this path to a maximal binary path $P = (x,v_1, v_2, \ldots, v_k, y)$ in $G$, where $v \in \{ v_1, v_2, \ldots, v_k \}$.
	Since $P$ is maximal, the degree of each of $x$ and $y$ in $G$ is different from $2$. By assumption, their degree cannot be $1$, so
    it is at least $3$. Hence, the degree of $x$ and $y$ in $G - \{ v_1, v_2, \ldots, v_k \}$ is at least 2. Furthermore, the degree of every other
    vertex in $G - \{ v_1, v_2, \ldots, v_k \}$ is the same as in $G$. Consequently, $G$ is a dumbbell with the handle $P$ and the
    plates corresponding to the two connected components of $G - \{ v_1, v_2, \ldots, v_k \}$. The minimum degree of the plates is at least 2.
    
    Assume now that every cut vertex in $G$ has degree at least 3. Let $v$ be such a vertex and
    $S_1, S_2, \ldots, S_t \subseteq V(G)$, $t \geq 2$, be the connected components of $G - v$.
    Note that every connected component has at least two vertices, as otherwise the unique
    vertex in a connected component would have degree 1 in $G$. If in some connected component 
    $v$ has a unique neighbour $u$, then $u$ is a cut vertex in $G$ and hence, by assumption, it has degree at least 3. Then $G$ is a dumbbell with the handle $P=(u,v)$ and the plates that correspond to the two connected components of the graph obtained from $G$ by deleting the cut edge $uv$. 
    If $v$ has more than two neighbours in each
    component, then $G$ is a $(H_1,H_2)$-dumbbell with the handle $P=(v)$ and the two plates $H_1 = G[\{v\} \cup S_1]$ and $H_2 = G[\{v\} \cup S_2 \cup \ldots \cup S_t]$ each with minimum degree at least 2.
\end{proof}

\section{Fractional domatic number}\label{sec:main}

In this section we prove our main \cref{th:main-intro}. 
As discussed in the introduction, if a graph $G$ has an isolated vertex, then its fractional domatic number is 1, and if the miminum degree of $G$ is at least one, then its fractional domatic number is at least 2. 
Moreover, it is easy to conclude that if the minimum degree of $G$ is exactly one, then its domatic number is exactly 2 (which follows e.g. from Theorem 3 \cite{GH21}).
Note that by \cref{lem:FDDisjUnion}, it is enough to prove \cref{th:main-intro} for connected graphs only. Furthermore, by \cref{lem:FDCycles}, the fractional domatic number of a $C_4$ is 2.
Hence, in this section, in order to prove \cref{th:main-intro} we will show

\begin{restatable}{theorem}{mainThm}
	\label{th:main}
	Let $G$ be a connected graph with $\delta(G) \geq 2$ that is different from $C_4$. 
	Then $\fdn(G) > 2$.
\end{restatable}

\noindent
Using \cref{lem:deg2decomposition}, we will split our analysis into two parts.
In \cref{sec:dumbbells}, we will deal with dumbbells, and in \cref{sec:2-conn} we will tackle 
2-connected graphs. In \cref{sec:proof-main} we will put everything together to prove 
\cref{th:main}. In the rest of this section we introduce necessary notation and prove some auxiliary results.
We start with useful properties of configurations.

\begin{observation}\label{obs:confAdditivity}
	If $\D$ is a $(k,s)$-configuration and 
	$\D'$ is a $(p,q)$-configuration of $G$, then
	$\D \cup \D'$ is a $(k+p, s+q)$-configuration of $G$.
\end{observation}

\begin{observation}\label{obs:2FDN}
	If a graph $G$ has fractional domatic numer greater than 2, then it admits a $(2k+1,k)$-configuration for any sufficiently large $k$.
\end{observation}
\begin{proof}
	By definition, $G$ admits a $(p,q)$-configuration such that $\fdn(G) = p/q > 2$.
	It follows that $p \geq 2q+1$. By removing $p-2q+1$ dominating sets from the $(p,q)$-configuration, we obtain a $(2q+1,q)$-configuration $\D$ of $G$.
	Since $\fdn(G) > 2$, $G$ has no isolated vertices and therefore, by \cref{obs:domComp}, for any minimal dominating set $D \subseteq V(G)$, its complement $\overline{D} = V(G) \setminus D$ is also a dominating set and thus they together form a $(2,1)$-configuration $\D'$. Now, for every $k > q$, from \cref{obs:confAdditivity}, by adding $k-q$ copies of $\D'$ to $\D$
	 we obtain a $(2k+1, k)$-configuration of $G$, as required.
\end{proof}

\noindent
Given a multiset $\D$ of dominating sets of $G$ and two distinct vertices
$x$ and $y$ in $G$ we define the following multisets
\begin{itemize}
	\item $\D_x := \{ D \in \D: x \in D, y \notin D \}$,
	\item $\D_y := \{ D \in \D: x \notin D, y \in D \}$,
	\item $\D_{xy} := \{ D \in \D: x \in D, y \in D \}$,
	\item $\D_{\overline{xy}} := \{ D \in \D: x \notin D, y \notin D \}$.
\end{itemize}

We say that a $(2r+1, r)$-configuration $\D$ of $G$ is \emph{$(x,y)$-nice} if each of 
$x$ and $y$ belongs to exactly $r$ sets in $\D$, and $\D_x$, $\D_y$, and $\D_{xy}$ are all nonempty.

\begin{lemma}\label{lem:earExtension}
	Let $G=(V,E)$ be a graph and let $P = (x, v_1, v_2, \ldots, v_s, y)$ be a binary path in $G$ 
	with at least two vertices. If the graph $H = G - \{ v_1, v_2, \ldots, v_s \}$ has a $(x,y)$-nice  $(2r+1,r)$-configuration, then $G$ has a $(2r+1,r)$-configuration.
\end{lemma}
\begin{proof}
	Let $\D = \{ D_1, D_2, \ldots, D_{2r+1} \}$ be a $(x,y)$-nice $(2r+1,r)$-configuration of $H$. 
	Since $\D$ is $(x,y)$-nice, we have $|\D_x| + |\D_{xy}| = |\D_y| + |\D_{xy}| = r$, $|\D_x| \geq 1$, $|\D_y| \geq 1$, and $|\D_{xy}| \geq 1$.
	The latter together with the fact that $\D_x \cup \D_y \cup \D_{xy} \cup \D_{\overline{xy}}$ is
	a partition of $\D$ implies a simple but important inequality that we will use later
	\begin{equation}\label{eq:empty}
		|\D_{\overline{xy}}| = |\D| - |\D_x| - |\D_y| - |\D_{xy}| = 2r+1- |\D_x| - |\D_y| - |\D_{xy}| 
		\leq 2r - |\D_x| - |\D_y|.
	\end{equation}
	We can assume that $P$ has at least one internal vertex, i.e. $\{ v_1, v_2, \ldots, v_s \} \neq \emptyset$, as otherwise $H$ would coincide with $G$ and the conclusion would trivially hold.
	
	Let $P'$ denote the path consisting of the internal vertices of $P$, i.e.  $P' = (v_1, v_2, \ldots, v_s)$, 
	$s\geq 1$.
	Furthermore, for every $\alpha \in \{ 0, 1, 2 \}$, we define 
	$R_{\alpha} = \{ v_i : i \equiv \alpha {\pmod {3}}, i \in [s] \}$.
	Notice that $R_1, R_2, R_3$ are pairwise disjoint sets.
	We are now ready to define $\D'$ by extending every dominating set of $H$ in $\D$ to a dominating set of $G$.
	Later we will show that $\D'$ is a $(2r+1,r)$-configuration of $G$. 
	We consider three cases.
	\begin{enumerate}
		\item If $s \equiv 0 {\pmod {3}}$, then 
		\begin{enumerate}
			\item $R_0$ dominates all vertices of $P'$ except $v_1$, and for every $D \in \D_x$
			we let $D' = D \cup R_0$;
			\item $R_1$ dominates all vertices of $P'$ except $v_s$, and for every $D \in \D_y$
			we let $D' = D \cup R_1$;
			\item $R_2$ dominates all vertices of $P'$, and for every $D \in \D_{\overline{xy}}$ we let
			$D' = D \cup R_2$;
			\item for every $D \in \D_{xy}$ we define $D' = D \cup R_0$.
		\end{enumerate}
		
		\item If $s \equiv 1 {\pmod {3}}$, then 
		\begin{enumerate}
			\item $R_0$ dominates all vertices of $P'$ except $v_1$, and for every $D \in D_x$
			we let $D' = D \cup R_0$;
			\item $R_1$ dominates all vertices of $P'$, and for every $D \in \D_{\overline{xy}}$
			we let $D' = D \cup R_1$;
			\item $R_2$ dominates all vertices of $P'$, except $v_s$, and for every $D \in \D_{y}$ we let
			$D' = D \cup R_2$;
			\item for every $D \in \D_{xy}$ we define $D' = D \cup R_0$.
		\end{enumerate}
		
		\item If $s \equiv 2 {\pmod {3}}$, then 
		\begin{enumerate}
			\item $R_0$ dominates all vertices of $P'$ except $v_1$ and $v_s$, and for every 
			$D \in D_{xy}$ we let $D' = D \cup R_0$;
			\item $R_1$ dominates all vertices of $P'$, and for every $D \in \D_x$
			we let $D' = D \cup R_1$;
			\item $R_2$ dominates all vertices of $P'$, and for every $D \in \D_y$ we let
			$D' = D \cup R_2$;
			
			\item first, we partition $\D_{\overline{xy}}$ arbitrary into two parts $\D_{\overline{xy}}^1$ and $\D_{\overline{xy}}^2$, 
			with $|\D_{\overline{xy}}^1| = \min \{ r - |\D_x|, |\D_{\overline{xy}}| \}$;
			second, for every $D \in \D_{\overline{xy}}^1$ we define $D' = D \cup R_1$, and 
			for every $D \in \D_{\overline{xy}}^2$ we define $D' = D \cup R_2$.
		\end{enumerate}
	\end{enumerate}
	
	We are now going to show that $\D'$ is a $(2r+1,r)$-configuration in $G$.
	It is clear from the construction that every set in $\D'$ is a dominating set in $G$ and $|\D'| = |\D| = 2r+1$.
	Hence, it remains to show that every vertex in $G$ belongs to at most $r$ sets in $\D'$.
	In fact, since $D' \setminus D \subseteq \{ v_1, v_2, \ldots, v_s \}$ for every $D \in \D$, we only have to
	prove that every vertex in $\{ v_1, v_2, \ldots, v_s \}$ is covered by at most $r$ sets in $\D'$.
	
	We will show the latter depending on the value of $s$ modulo 3.
	The cases $s \equiv 0 {\pmod {3}}$ and $s \equiv 1 {\pmod {3}}$ are similar, 
	so we consider only one of them.
	
	Let $s \equiv 0 {\pmod {3}}$.
	By definition, $R_0, R_1$, and $R_2$ are pairwise disjoint, and therefore we can treat vertices 
	$v_1, v_2, \ldots, v_s$ depending on which of the three sets they belong to.
	By construction, all vertices in $R_0$ belong to $|\D_x|+|\D_{xy}| \leq r$ dominating sets.
	Similarly, all vertices of $R_1$ belong to $|\D_y| \leq r$ dominating sets.
	Finally, all vertices in $R_2$ belong to $|\D_{\overline{xy}}|$ dominating sets and 
	$$
	|\D_{\overline{xy}}| = |\D| - |\D_x| - |\D_y| - |\D_{xy}| = 2r+1 - (|\D_x| + |\D_y| + |\D_{xy}|) \leq r,
	$$
	where the latter inequality follows from the fact that $|\D_x| + |\D_y| + |\D_{xy}| \geq r + 1$.

	Let now $s \equiv 2 {\pmod {3}}$.
	By construction, all vertices in $R_0$ belong to $|\D_{xy}| \leq r$ dominating sets.
	Similarly, all vertices of $R_1$ belong to 
	$|\D_x| + |D_{\overline{xy}}^1| = |\D_x| + \min \{ r - |\D_x|, |\D_{\overline{xy}}| \} \leq r$ dominating sets.
	Finally, all vertices in $R_2$ belong to $|\D_y| + |D_{\overline{xy}}^2|$ sets and
	\begin{equation*}
		\begin{split}
			|\D_y| + |D_{\overline{xy}}^2| &= |\D_y| + |\D_{\overline{xy}} \setminus D_{\overline{xy}}^1| \\
			&= |\D_y| + |\D_{\overline{xy}}| - \min\{ r - |\D_x|, |\D_{\overline{xy}}| \} \\
			&= \max\{ |\D_y| + |\D_{\overline{xy}}| - (r - |\D_x|), |\D_y| \}\\
			&\leq \max\{ |\D_y| + 2r - |\D_x| - |\D_y| - (r - |\D_x|), |\D_y| \}\\
			&= \max\{ r, |\D_y| \} = r,
		\end{split}
	\end{equation*}
	where in the inequality we used (\ref{eq:empty}).	
\end{proof}

\subsection{Dumbbells}
\label{sec:dumbbells}

\subsubsection{$(C_4,C_4)$-dumbbells}

\begin{lemma}
	\label{lem:C4C4dumbbell}
	Let $G$ be a $(C_4,C_4)$-dumbbell. Then $G$ admits a $(7,3)$-configuration.
\end{lemma}
\begin{proof}
	Let $a,b,c, d_1, d_2, \ldots, d_s, e, f, g$ be the vertices of $G$, where $(a,b,d_1,c,a)$ and $(d_s,e,g,f,d_s)$ 
	are the $C_4$ plates of the dumbbell, and $P=(d_1, d_2, \ldots, d_s)$, $s \geq 1$ is its handle.
	
	We define three pairwise disjoint sets $R_{\alpha} = \{ d_i : i \equiv \alpha {\pmod {3}}, i \in [s] \}$, $\alpha \in \{ 0, 1, 2 \}$. It is straightforward to check that, depending on the value of $s$ modulo 3, each of the three families of sets in \cref{tab:C4C4} is a $(7,3)$-configuration of $G$.
	
	\setlength{\tabcolsep}{0.5em} 
	{\renewcommand{\arraystretch}{1.5}
	\begin{table}[ht]
		\centering
		\begin{tabular}{|p{40mm}|p{40mm}|p{40mm}|}
			\hline
			$s \equiv 0 {\pmod {3}}$ & $s \equiv 1 {\pmod {3}}$ & $s \equiv 2 {\pmod {3}}$  \\ \hline	
			$R_0 \cup \{b, c, e\}$ \newline
			$R_0 \cup R_1 \cup \{c, f\}$ \newline
			$R_0 \cup R_1 \cup \{b, e\}$ \newline
			$R_1 \cup \{b, e, f\}$ \newline
			$R_2 \cup \{a, g\}$ \newline
			$R_2 \cup \{a, g\}$ \newline
			$R_2 \cup \{a, g\}$ 
			&  
			$R_0 \cup \{a, c, g\}$ \newline
			$R_0 \cup \{b, c, g\}$ \newline
			$R_1 \cup \{b, g\}$ \newline
			$R_1 \cup \{c, e\}$ \newline
			$R_1 \cup \{b, f\}$ \newline
			$R_2 \cup \{a, e, f\}$ \newline
			$R_2 \cup \{a, e, f\}$ 
			& 
			$R_0 \cup \{a, c, e, f\}$ \newline
			$R_1 \cup \{b, g\}$ \newline
			$R_1 \cup \{b, g\}$ \newline
			$R_1 \cup \{c, g\}$ \newline
			$R_2 \cup \{a, e\}$ \newline
			$R_2 \cup \{a, e\}$ \newline
			$R_2 \cup \{b, c, f\}$ \\ \hline
		\end{tabular}
	\caption{$(7,3)$-configurations of $G$ depending on the value of $s$ modulo 3.}
	\label{tab:C4C4}
	\end{table}	
	}
\end{proof}

\subsubsection{$(C_4,H)$-dumbbells}

Let $S$ be an $n$-element set. A partition of $S$ into $k \leq n$ sets is \emph{balanced} if the cardinalities of any two parts in the partition differ at most by one. Clearly, from the definition, the cardinality of any part is either 
$\lceil n/k \rceil$ or $\lfloor n/k \rfloor$.

\begin{lemma}
	\label{lem:C4Hdumbbell}
	Let $G$ be a $(C_4, H)$-dumbbell. If $H$ admits a $(2k+1,k)$-configuration, for $k \geq 3$, then $G$ does too.
\end{lemma}
\begin{proof}
	Let $(b,a,c,d_1,b)$ be the $C_4$ plate of the dumbbell and $P=(d_1,d_2, \ldots, d_s)$, $s \geq 1$ be the handle of
	the dumbbell. 
	Let $\D$ be a $(2k+1,k)$-configuration of $H$, $\D_{d_s}$ be the family of dominating sets in $\D$ that contain $d_s$,
	and $\D_{\overline{d_s}}$ be the family of dominating sets that do not contain $d_s$, i.e. $\D_{\overline{d_s}} = \D \setminus \D_{d_s}$.
	Without loss of generality we assume that $|\D_{d_s}| = k$.
	
	We break down the analysis into three cases depending on the value of $s$ modulo 3. 
	As before we define three pairwise disjoint sets $R_{\alpha} = \{ d_i : i \equiv \alpha {\pmod {3}}, i \in [s] \}$, $\alpha \in \{ 0, 1, 2 \}$.
	
	Suppose first that $s \equiv 0 {\pmod {3}}$. Let $\D_{d_s} = \D_{d_s}^1 \cup \D_{d_s}^2$ be an arbitrary balanced partition of $\D_{d_s}$, and let $\D_{\overline{d_s}} = \D_{\overline{d_s}}^1 \cup \D_{\overline{d_s}}^2 \cup \D_{\overline{d_s}}^3$ be an arbitrary balanced partition of $\D_{\overline{d_s}}$. 
	Furthermore, if $k=3$ we assume that $|\D_{d_s}^1| = 1$,  $|\D_{d_s}^2| = 2$, and 
	$|\D_{\overline{d_s}}^1| = |\D_{\overline{d_s}}^2| = 1$ and $|\D_{\overline{d_s}}^3| = 2$.
	We define the desired $(2k+1,k)$-configuration $\F$ of $G$ as the union of the following five multisets:
	\begin{enumerate}
		\item $\F_1 = \left\{ R_0 \cup \{b, c\} \cup D ~|~ D \in \D_{d_s}^1 \right\}$
		\item $\F_2 = \left\{ R_0 \cup R_2 \cup \{a\} \cup D ~|~ D \in \D_{d_s}^2 \right\}$
		\item $\F_3 = \left\{  R_2 \cup \{a\} \cup D ~|~ D \in \D_{\overline{d_s}}^1 \right\}$
		\item $\F_4 = \left\{  R_1 \cup \{b\} \cup D ~|~ D \in \D_{\overline{d_s}}^2 \right\}$
		\item $\F_5 = \left\{  R_1 \cup \{c\} \cup D ~|~ D \in \D_{\overline{d_s}}^3 \right\}$
	\end{enumerate}
	It is easy to check that all sets in $\F$ are dominating sets of $G$.
	Furthermore, notice that $\F$ is obtained by extending sets in $\D$, so the two configurations have the same cardinality $2k+1$. It remains to verify that every vertex of $G$ is covered by at most $k$ sets in $\F$. 
	This is clearly the case for any vertex of $H$ that is different from $d_s$. For $d_s$, we observe that it belongs to $R_0$, and $R_0$ is added only to the sets from $\D_{d_s}$. Hence $d_s$ is covered 
	in $\F$ by the same number of sets as in $\D$, i.e. by exactly $k$ sets. 
	Similarly, all the other vertices in $R_0$ are covered by exactly $k$ sets.
	Vertex $a$ and each of the vertices in $R_2$ are covered by $|\F_2| + |\F_3| = |\D_{d_s}^2| + |\D_{\overline{d_s}}^1|$ sets. This number is equal to $k$ if $k = 3$ and it is at most 
	$\lceil k/2 \rceil + \lceil (k+1)/3 \rceil \leq k$ if $k \geq 4$.
	Similarly, vertex $b$ is covered by $|\F_1| + |\F_4| = |\D_{d_s}^1| + |\D_{\overline{d_s}}^2|$ sets and
	vertex $c$ is covered by $|\F_1| + |\F_5| = |\D_{d_s}^1| + |\D_{\overline{d_s}}^3|$ sets. As before, in both cases, the number of covering sets is at most $k$. Finally, every vertex in $R_1$ is covered by 
	$|\F_4| + |\F_5| = |\D_{\overline{d_s}}^2| + |\D_{\overline{d_s}}^3|$ sets, which is equal to $k$ if $k=3$
	and does not exceed $2 \lceil (k+1)/3 \rceil \leq k$ if $k \geq 4$.
	
	Assume now that $s \equiv 1 {\pmod {3}}$. Let $\D_{d_s} = \D_{d_s}^1 \cup \D_{d_s}^2 \cup \D_{d_s}^3$ be an arbitrary balanced partition of $\D_{d_s}$, and let $\D_{\overline{d_s}} = \D_{\overline{d_s}}^1 \cup \D_{\overline{d_s}}^2$ be an arbitrary balanced partition of $\D_{\overline{d_s}}$.
	We define the desired $(2k+1,k)$-configuration $\F$ of $G$ as the union of the following five multisets:
	\begin{enumerate}
		\item $\F_1 = \left\{ R_1 \cup \{a\} \cup D ~|~ D \in \D_{d_s}^1 \right\}$
		\item $\F_2 = \left\{ R_1 \cup \{b\} \cup D ~|~ D \in \D_{d_s}^2 \right\}$
		\item $\F_3 = \left\{ R_1 \cup \{c\} \cup D ~|~ D \in \D_{d_s}^3 \right\}$
		\item $\F_4 = \left\{ R_0 \cup \{b,c\} \cup D ~|~ D \in \D_{\overline{d_s}}^1 \right\}$
		\item $\F_5 = \left\{ R_2 \cup \{a\} \cup D ~|~ D \in \D_{\overline{d_s}}^2 \right\}$
	\end{enumerate}
	Following similar reasoning as in the previous case, on can check that $|\F| = |\D|$, the sets in $\F$ are dominating sets of $G$, and all vertices of $H$ are covered by the same number of sets in $\F$ as in $\D$. 
	Furthermore, every vertex in $R_1$ is covered by exactly $k$ sets, and every other vertex outside $V(H)$ is covered by at most $\lceil k/3 \rceil + \lceil (k+1)/2 \rceil \leq k$ sets in $\F$. Hence, $\F$ is a
	$(2k+1,k)$-configuration of $G$.
	
	Finally, assume that $s \equiv 2 {\pmod {3}}$. Let $\D_{\overline{d_s}} = \D_{\overline{d_s}}^1 \cup \D_{\overline{d_s}}^2 \cup \D_{\overline{d_s}}^3$ be an arbitrary balanced partition of $\D_{\overline{d_s}}$.
	We define the desired $(2k+1,k)$-configuration $\F$ of $G$ as the union of the following four multisets:
	\begin{enumerate}
		\item $\F_1 = \left\{ R_2 \cup \{a\} \cup D ~|~ D \in \D_{d_s} \right\}$
		\item $\F_2 = \left\{ R_1 \cup \{b\} \cup D ~|~ D \in \D_{\overline{d_s}}^1 \right\}$
		\item $\F_3 = \left\{ R_1 \cup \{c\} \cup D ~|~ D \in \D_{\overline{d_s}}^2 \right\}$
		\item $\F_4 = \left\{ R_0 \cup \{b,c\} \cup D ~|~ D \in \D_{\overline{d_s}}^3 \right\}$
	\end{enumerate}
	Using arguments similar to those used in the previous two cases, it is not hard to verify that $\F$ is indeed a $(2k+1,k)$-configuration of $G$.
\end{proof}

\subsubsection{$(H_1,H_2)$-dumbbells}

\begin{lemma}
	\label{lem:H1H2dumbbell}
	Let $G$ be a $(H_1, H_2)$-dumbbell such that $\delta(H_1) \geq 2$ and $\delta(H_2) \geq 2$.
	If $H_1$ admits a $(2r+1,r)$-configuration and $H_2$ admits a $(2k+1,k)$-configuration, where $r \geq k \geq 2$, then $G$ admits a $(2r+1,r)$-configuration.
\end{lemma}
\begin{proof}
	We start by observing that since $H_2$ admits a $(2k+1,k)$-configuration, it also admits a $(2r+1,r)$-configuration. Indeed, let $D$ be a minimal dominating set in $H_2$. 
	As $H_2$ does not contain isolated vertices, by \cref{obs:domComp}, 
	$\overline{D} = V(H_2) \setminus D$ is also a dominating set in $H_2$. 
	Hence, $\{ D, \overline{D} \}$ is a $(2,1)$-configuration of $H_2$. 
	Therefore, \cref{obs:confAdditivity} implies that by extending a $(2k+1,k)$-configuration of $H_2$ with $r-k$ copies of $\{ D, \overline{D} \}$ we obtain a $(2r+1,r)$-configuration of $H_2$.
	
	Let $\D = \{ D_1, D_2, \ldots, D_{2r+1} \}$ be a $(2r+1,r)$-configuration of $H_1$ and 
	$\D' = \{ D_1', D_2', \ldots, D_{2r+1}' \}$ be a $(2r+1,r)$-configuration
	of $H_2$. Let also $P$ be the handle of $G$. Note that $H_1$ and $H_2$ either have exactly one vertex in common (if $|V(P)|  = 1$), or they are vertex disjoint (if $|V(P)| \geq 2$).
	Observe that in either case, for any $D_i \in \D$ and $D_j' \in \D'$, $i,j \in [2r+1]$, the set $D_i \cup D_j'$ is a dominating set in $H_1 \cup H_2$.
	
	Suppose first that $H_1$ and $H_2$ have one vertex in common, which we denote by $x$.
	Without loss of generality assume that the dominating sets in $\D$ are indexed in such a way
	that $x$ belongs to $D_i$ if and only if $i \in [t]$ for some $t \leq r$. Similarly, assume that
	the sets in $\D'$ are indexed in such a way that $x \in D_i'$ if and only if $i \in [t']$ for some $t' \leq r$. Then it is easy to see that $\{ D_i \cup D_i' ~|~ i \in [2r+1] \}$ is a $(2r+1,r)$-configuration of $G = H_1 \cup H_2$. 
	
	Suppose now that $H_1$ and $H_2$ are vertex disjoint and the handle $P = (x, v_1, v_2, \ldots, v_s, y)$ has at least two vertices: $x \in V(H_1)$ and $y \in V(H_2)$. 
	By definition, $P$ is a binary path in $G$ and $G - \{v_1, v_2, \ldots, v_s \} = H_1 \cup H_2$. 
	In this case, the statement immediately follows from \cref{lem:earExtension} if we 
	show that $H_1 \cup H_2$ admits a $(x,y)$-nice $(2r+1,r)$-configuration. 
	To prove the latter, we can assume that $x$ belongs to exactly $r$ sets in $\D$. 
	Indeed, otherwise we can add $x$ to some sets in $\D$ that do not contain $x$ to ensure
	the property. Similarly, we can assume that $y$ belongs to exactly $r$ sets in $\D'$. Without loss of generality, suppose that $x \in D_i$ and $y \in D_i'$ if and only if $i \in [r]$. Then 
	$$
		\D'' = \{ D_1 \cup D_{r+1}' \} \cup \{ D_1' \cup D_{r+1} \} \cup \{ D_i \cup D_i' ~|~ i \in [2r+1] \setminus \{1,r+1\} \}
	$$
	is a $(2r+1,r)$-configuration of $H_1 \cup H_2$ and each of $x$ and $y$ belongs to exactly $r$ sets in $\D''$, and $\D_x''$, $\D_y''$, and $\D_{xy}''$ are all nonempty, i.e.
	$\D''$ is $(x,y)$-nice. This completes the proof.
\end{proof}

\subsection{2-connected graphs}
\label{sec:2-conn}

The main goal of this subsection is to prove that the fractional domatic number
of any 2-connected graph $G$, which is distinct from $C_4$, is more than 2.
To this end we will employ ear decompositions. 
An \emph{ear} in a graph is either a cycle with at least 3 vertices or a path.

\begin{definition}
	An \emph{open ear decomposition} of $G=(V,E)$ is a sequence $P_1, P_2, \ldots, P_t$ of ears in $G$
	such that 
	\begin{enumerate}
		\item $E(P_1) \cup E(P_2) \cup \ldots \cup E(P_t)$ is a partition of $E$; 
		\item the first ear $P_1$ in the sequence is a cycle; and
		\item for every $i = 2, 3, \ldots, t$, $P_i$ is a path and exactly two vertices of $P_i$,
		which are the endpoints of $P_i$, belong to the earlier ears.
	\end{enumerate}
\end{definition}

The following theorem is a classical result due to H.~Whitney.

\begin{theorem}[Whitney, \cite{Whi92}]\label{th:ear2connected} 
	A graph is $2$-connected if and only if it has an open ear decomposition.
\end{theorem}

Our strategy is as follows. We will show that if the fractional domatic number
of a graph $G$ is more than 2, then after adding an open ear $P=(x,\ldots, y)$, to $G$ the fractional domatic number of the resulting graph $G'$ stays strictly above 2.
To prove this, we will show (\cref{lem:confExt}) that $G$ admits a $(x,y)$-nice $(2r+1,r)$-configuration. The result will then follow from \cref{lem:earExtension} as $P$ is a binary path in $G'$. 

Using this result and \cref{th:ear2connected} we will show (\cref{th:2con}), by induction on the number of ears in an ear decomposition, that the fractional domatic number is above 2 as long as it is above 2 for the first ear in the ear decomposition. Since, by \cref{lem:FDCycles}, the fractional domatic number of any cycle that is different from $C_4$ is more than 2 and any cycle can start an ear decomposition, the above would work for any 2-connected graph that has a cycle different from $C_4$. As we show in the next two lemmas, in the case when all cycles in a 2-connected graph are $C_4$s, the graph has very simple structure and its fractional domatic number is more than 2.

\begin{lemma}\label{lem:onlyC4}
	Let $G$ be a 2-connected graph with $n \geq 4$ vertices that contains no 
	(not necessarily induced) cycles of length other than 4. Then $G = K_{2,n-2}$.
\end{lemma}
\begin{proof}
	We prove the statement by induction on the minimum number of ears in an open ear decomposition. 
	Let $P_1, P_2, \ldots, P_k$ be an open ear decomposition of $G$ with the minimum number of ears.
	If $k=1$, then by the assumption $G = C_4 = K_{2,2}$ and the statement holds. 
	Let now $k \geq 2$ and let $G'$ be the 2-connected graph with the ear decomposition 
	$P_1, P_2, \ldots, P_{k-1}$. Since $G'$ is obtained from $G$ by removing some edges and/or vertices
	(namely, the edges and the internal vertices of $P_k$), $G'$ has no cycles of length other than 4.
	Hence, by the induction hypothesis, $G' = K_{2,p}$ for some $p \geq 2$.  
	
	Now, if $P_k$ consists of a single edge $xy$, then $x$ and $y$ are non-adjacent in $G'$,
	and therefore they belong to the same part of the $K_{2,p}$. But then $G$ contains a cycle on 3 vertices.
	Hence we assume that $P_k$ has at least one internal vertex and denote 
	$P_k=(x, v_1, v_2, \ldots, v_t, y)$, where $t \geq 1$. 
	If the number $t$ of the internal vertices in $P_k$ is at least 2, then it is easy to see that 
	$G$ contains a cycle of length more than 4. Hence $t = 1$. If $x$ and $y$ belong to the different parts
	of $K_{2,p}$ or they are in the part with $p \geq 3$ vertices, then it is easy to check that $G$ 
	contains a cycle on at least 5 vertices. This contradiction shows that $x$ and $y$ must belong
	to the part of $K_{2,p}$ that contains 2 vertices, and hence $G=K_{2,p+1}$.
\end{proof}

\begin{lemma}\label{lem:FDcompleteBip}
	For every $p \geq 2$, $K_{2,p}$ admits a $(3p-2, p)$-configuration, and therefore $\fdn(K_{2,p}) \geq \frac{3p-2}{p}$.
\end{lemma}
\begin{proof}
	Let $a_1, a_2$ and $b_1, b_2, \ldots, b_p$ be the vertices of the $K_{2,p}$ such that both $a_1$ and $a_2$ are
	adjacent to every vertex $b_i$, $i \in [p]$. 
	It is straightforward to verify that the following family of dominating sets is a $(3p-2, p)$-configuration:
	\begin{enumerate}
		\item the $p$ sets $\{ a_1, b_i \}$, $i \in [p]$;
		\item the $p$ sets $\{ a_2, b_i \}$, $i \in [p]$; and
		\item $p-2$ copies of $\{ b_1, b_2, \ldots, b_p \}$.
	\end{enumerate}
\end{proof}

\begin{observation}\label{obs:addDominatingSet}
	Let $G=(V,E)$ be a graph, let $S$ be a dominating set in $G$, and
	let $\D$ be a $(2k+1,k)$-configuration of $G$.
	Then the multiset $\D'$ that consists of the set $S$ and
	two copies of every element in $\D$ is a $(2r+1,r)$-configuration of $G$,
	where $r=2k+1$.
\end{observation}
\begin{proof}
	By construction, $|\D'| = 2|\D|+1 = 2r+1$. Moreover, every vertex in $G$ belongs to
	at most $2k$ sets in $\D' \setminus \{ S \}$ and therefore every vertex belongs to at most 
	$r=2k+1$ sets of $\D'$.
\end{proof}

\begin{lemma}\label{lem:confExt}
	Let $G=(V,E)$ be a graph and let $x,y$ be two distinct vertices in $G$.
	If $G$ has a $(2k+1,k)$-configuration $\D$,
	then, for some $r \leq 2k+1$, $G$ has a $(x,y)$-nice $(2r+1,r)$-configuration $\D'$.
\end{lemma}
\begin{proof}
	First, let us assume that $1 \leq |\D_{xy}| < k$.
	In this case, the desired $(2r+1,r)$-configuration $\D'$, where $r=k$, 
	is obtained from $\D$ by extending some of the sets in $\D$ as follows:
	\begin{enumerate}
		\item add $x$ to $k - |\D_x| - |\D_{xy}|$ sets of $\D_{\overline{xy}}$;
		\item add $y$ to some other $k - |\D_y| - |\D_{xy}|$ sets of $\D_{\overline{xy}}$.
	\end{enumerate}
	Notice that such a modification can always be done, because 
	$$
	(k - |\D_x| - |\D_{xy}|) + (k - |\D_y| - |\D_{xy}|) \leq
	(k - |\D_x| - |\D_{xy}|) + (k - |\D_y|) + 1 = |\D_{\overline{xy}}|.
	$$
	Clearly, every new set remains dominating in $G$, and each of $x$ and $y$
	belongs to exactly $k$ sets of $\D'$. Moreover, since $|\D_{xy}| < k$,
	we have that both $\D_x'$ and $\D_y'$ are nonempty. Hence, $\D'$ is $(x,y)$-nice.

	Suppose now $|\D_{xy}| = 0$. Then we define $\D'$ to be 
	the multiset obtained from $\D$ by duplicating every 
	element in $\D$ and adding the dominating set $V$.
	By \cref{obs:addDominatingSet}, the multiset $\D'$ is a 
	$(2r+1,r)$-configuration of $G$, where $r =  2k+1$. 
	Moreover, $|\D_{xy}'| = 1 < r$,
	and therefore, by the above argument, $\D'$ can be turned into
	the desired $(x,y)$-nice $(2r+1,r)$-configuration.
	
	Similarly, if $|\D_{xy}| = k$, we define $\D'$ to be the multiset 
	obtained from $\D$ by duplicating every element in $\D$ 
	and adding the set $V \setminus \{ x,y \}$, which is dominating
	because every vertex in $G$ has degree at least two as $\fdn(G) \geq \frac{2k+1}{k} > 2$.
	As before, $\D'$ is a $(2r+1,r)$-configuration with $r =  2k+1$. Moreover,
	$1 \leq |\D_{xy}'| = 2k < r$, and hence, by the above argument the desired
	$(x,y)$-nice $(2r+1,r)$-configuration can be obtained from $\D'$.
\end{proof}

\begin{theorem}
	\label{th:2con}
	Let $G$ be a 2-connected graph, which is not a $C_4$. Then $\fdn(G) > 2$.
\end{theorem} 
\begin{proof}
	If $G$ is equal to $K_{2,p}$ for some $p \geq 3$, then $\fdn(G) \geq \frac{7}{3} > 2$  by \cref{lem:FDcompleteBip}.
	Hence, by \cref{lem:onlyC4}, we assume that $G$ contains a cycle $C$ of length 3 or more than 4. 
	It is known that for any cycle in a 2-connected graph there is an open ear decomposition that starts
	with this cycle (see e.g. \cite[Theorem 4.2.8]{West01}). 
	Let $P_1, P_2, \ldots, P_t$, $t \geq 1$ be an open ear decomposition of $G$ with $P_1 = C$.
	We will show by induction on $i \in [t]$ that the graph $G_i$ with the open ear decomposition 
	$P_1, P_2, \ldots, P_i$ has fractional domatic number greater than 2. The theorem will then follow for $i = t$.
	
	In the base case $i=1$, the claim follows from \cref{lem:FDCycles}. 
	Assume therefore that $i \geq 2$ and the statement holds for $G_{i-1}$. 
	Let $P_i = (x, v_1, v_2, \ldots, v_q, y)$ and note that $P_i$ is a binary path in $G_i$.
	Since $\fdn(G_{i-1}) > 2$, by \cref{obs:2FDN} there exists a $(2k+1,k)$-configuration of $G_{i-1}$, for some $k \geq 1$, and hence, by \cref{lem:confExt}, $G_{i-1}$ admits a $(x,y)$-nice $(2r+1,r)$-configuration for some $r \leq 2k+1$.
	Therefore, by \cref{lem:earExtension}, $G_i$ has a $(2r+1,r)$-configuration implying $\fdn(G) \geq \frac{2r+1}{r} > 2$.
\end{proof}

\subsection{Putting everything together: proof of \cref{th:main}}
\label{sec:proof-main}

In this section we combine the facts presented in the paper to prove our main result

\mainThm*
\begin{proof}
	We will prove the statement by induction on the number of vertices in $G$. 
	Clearly, the statement holds for graphs with at most two vertices, as in this case the minimum degree is less than 2. 
	Suppose that $|V(G)| > 2$ and 
	the statement holds for all graphs with less than $|V(G)|$ vertices.
	
	Since $\delta(G) \geq 2$, by \cref{lem:deg2decomposition}, $G$ is either a 2-connected graph
	or a $(H_1,H_2)$-dumbbell for some connected graphs $H_1$ and $H_2$ with $\delta(H_1) \geq 2$
	and $\delta(H_2) \geq 2$. In the former case, the result holds by \cref{th:2con}. Hence assume that $G$ is
	a $(H_1,H_2)$-dumbbell.
	
	If both $H_1$ and $H_2$ are $C_4$s, then the result follows from \cref{lem:C4C4dumbbell}.
	If exactly one of the graphs, say $H_1$, is $C_4$, and  graph $H_2$ is different from $C_4$ and $\delta(H_2) \geq 2$, then $\fdn(H_2) > 2$ by the induction hypothesis,
	and hence, by \cref{obs:2FDN}, graph $H_2$ admits a $(2k+1,k)$-configuration for some $k \geq 3$. Then the statement follows from \cref{lem:C4Hdumbbell}. Finally, if both $H_1$ and $H_2$ are different from $C_4$, then by the induction hypothesis, both of them have fractional domatic number more than 2 and therefore they admit $(2r+1,r)$- and $(2k+1,k)$-configurations respectively for some positive $r$ and $k$. The statement then follows from \cref{lem:H1H2dumbbell}.
\end{proof}

\section{Conclusion}
\label{sec:conclusion}

In this paper we characterized graphs with fractional domatic number 2.
In order to do this, we showed that any connected graph of minimum degree at least two that is
different from $C_4$ has fractional domatic number more than 2.
While our proof does not bound the fractional domatic number away from 2, we did not find
any graph with this parameter being strictly between 2 and 7/3. 
In fact we believe that there are no such graphs

\begin{conjecture}
	If the fractional domatic number of a graph is greater than 2, then it is at least 7/3.
\end{conjecture}

\noindent
In our approach the only obstacle to proving this conjecture is \cref{lem:confExt} which 
guarantees that given a $(2k+1,k)$-configuration one can always find a $(x,y)$-nice $(2r+1,r)$-configuration
for some $r \leq 2k+1$, but does not guarantee that one can always find such a configuration with $r=k$, which would be enough to settle the conjecture.

\paragraph{Acknowledgments.} We thank Chun-Hung Liu and Sarosh Adenwalla for their feedback on an earlier draft of the paper.
We are grateful to the anonymous referees for their thorough reading of the paper and the detailed comments, which improved the presentation of the paper.

\bibliographystyle{plainurl}
\bibliography{bibliography}

\begin{thebibliography}{10}

\bibitem{WMCRP16}
Waseem Abbas, Magnus Egerstedt, Chun-Hung Liu, Robin Thomas, and Peter Whalen.
\newblock Deploying robots with two sensors in ${K}_{1,6}$-free graphs.
\newblock {\em Journal of Graph Theory}, 82(3):236--252, 2016.

\bibitem{CH75}
Ernest Cockayne and Stephen Hedetniemi.
\newblock Optimal domination in graphs.
\newblock {\em IEEE Transactions on circuits and systems}, 22(11):855--857,
  1975.

\bibitem{FYK00}
Satoshi Fujita, Masafumi Yamashita, and Tiko Kameda.
\newblock A study on $r$-configurations---a resource assignment problem on
  graphs.
\newblock {\em SIAM Journal on Discrete Mathematics}, 13(2):227--254, 2000.

\bibitem{GH21}
Wayne Goddard and Michael~A Henning.
\newblock Fractional domatic, idomatic, and total domatic numbers of a graph.
\newblock In {\em Structures of Domination in Graphs}, pages 79--99. Springer,
  2021.

\bibitem{MW05}
Thomas Moscibroda and Roger Wattenhofer.
\newblock Maximizing the lifetime of dominating sets.
\newblock In {\em 19th IEEE International Parallel and Distributed Processing
  Symposium}, pages 8--pp. IEEE, 2005.

\bibitem{Ore62}
Oystein Ore.
\newblock Theory of graphs.
\newblock Colloquium Publications. American Mathematical Society, 1962.

\bibitem{SU11}
Edward~R Scheinerman and Daniel~H Ullman.
\newblock {\em Fractional graph theory: a rational approach to the theory of
  graphs}.
\newblock Courier Corporation, 2011.

\bibitem{Suo06}
Jukka Suomela.
\newblock Locality helps sleep scheduling.
\newblock In {\em Working Notes of the Workshop on World-Sensor-Web: Mobile
  Device-Centric Sensory Networks and Applications}, pages 41--44. Citeseer,
  2006.

\bibitem{West01}
Douglas~B West.
\newblock {\em Introduction to graph theory}, volume~2.
\newblock Prentice hall Upper Saddle River, 2001.

\bibitem{Whi92}
Hassler Whitney.
\newblock Non-separable and planar graphs.
\newblock In {\em Hassler Whitney Collected Papers}, pages 37--59. Springer,
  1992.

\end{thebibliography}

\end{document}